\documentclass[11pt]{amsart}

\title[Quantum Cohomology of Fano Bundles]{A Reconstruction Theorem for Quantum Cohomology of Fano Bundles on $\PP^n$}

\usepackage{pdflscape}
\usepackage[all]{xy}
\usepackage{algpseudocode}
\usepackage{cite}
\usepackage{url}
\usepackage[obeyDraft,backgroundcolor=yellow,textsize=tiny]{todonotes}

\newcommand{\PP}{\mathbb{P}}
\newcommand{\RR}{\mathbb{R}}
\newcommand{\OO}{\mathcal{O}}
\newcommand{\C}{\mathbb{C}}
\newcommand{\z}{\mathbb{Z}}
\newcommand{\ev}{\mathrm{ev}}
\newcommand{\pf}{\mathrm{pf}}
\newcommand{\Ind}[1]{\mathrm{I}\left({#1}\right)}
\newcommand{\V}[1]{V_{({#1})}}
\newcommand{\term}[1]{\mathrm{Term}\left({#1}\right)}
\renewcommand{\deg}{\operatorname{deg}}
\newcommand{\PD}{\operatorname{PD}}
\long\def\blankfootnotetext#1{\begingroup\def\thefootnote{\fnsymbol{footnote}}\footnotetext{#1}\endgroup}
\newtheorem{theorem}{Theorem}
\newtheorem{lem}{Lemma}

\newtheorem{cor}{Corollary}
\theoremstyle{definition}
\newtheorem{ass}{Assumption}
\newtheorem{rem}{Remark}

\begin{document}
\author[A. ~Strangeway]{Andrew Strangeway}
\address{Department of Mathematics\\Imperial College London\\London, SW$7$\ $2$AZ\\UK}
\email{a.strangeway09@imperial.ac.uk}
\blankfootnotetext{2010 \emph{Mathematics Subject Classification}: 14N35 (Primary); 53D45, 14J45  (Secondary).}
\begin{abstract}
We present a reconstruction theorem for Fano vector bundles on projective space which recovers the small quantum cohomology for the projectivisation of the bundle from a small number of low-degree Gromov--Witten invariants. 
We provide an extended example in which we calculate the quantum cohomology of a certain Fano \mbox{9-fold} and deduce from this, using the quantum Lefschetz theorem, the quantum period sequence for a Fano \mbox{3-fold} of Picard rank 2 and degree 24. This example is new, and is important for the Fanosearch programme.
\end{abstract}

\maketitle
\section{Introduction}
We say that a vector bundle $E$ is Fano if the total space of the projectivisation, $X:=\PP(E)$, is a Fano manifold. In this paper we study \hbox{rank-$r$} Fano vector bundles on \hbox{$n$-dimensional} projective space $\PP^n$.
It is well known that the ring $H^\bullet(X;\z)$ has two integer generating classes, $p$ and $\xi_E$, given by the tautological bundle on $\PP^n$ and the relative tautological bundle on the projectivisation $\PP(E)$, respectively.
These are canonical up to a choice of normalisation of $E$.

The Mori cone of $X$ has precisely two extremal rays $R_1, R_2\subset\overline{NE}_X\subset H_2(X;\RR)$.
We make the assumption on the bundle $E$ ({\bfseries Assumption~1}) that the primitive generators of the rays $R_1$, $R_2$ form an integer basis of $H_2(X;\z)$.
We label this basis by $A_1$, $A_2$.

Without loss of generality we may take $A_2$ to be the class of a line in the fibre. 
The representative for $A_1$ is less clear and indeed it is possible that it may not be represented by an algebraic curve (cf.~\cite[Question~3~and~Example~4]{JK_example}. 
It is not clear whether for Fano bundles Assumption~1 always holds, although there exist counter examples for Fano varieties in general.
 
It follows from Assumption~1 that we can normalise the bundle such that the basis $A_1,A_2$ of $H_2(X;\z)$ is Poincar\'e dual to the basis $\widehat{p},\widehat{\xi_E}$ of $H^2(X;\z)$. 
Here $\widehat{\alpha}$ denotes the dual of $\alpha\in H^\bullet(X)$  with respect to the pairing $(\alpha,\beta)=\int_{X}{\alpha\cup\beta}$ on cohomology. 
In the following we shall assume that $E$ is normalised in this manner and denote the tautological class by $\xi$ without reference to the corresponding bundle.

Given Assumption~1 we determine the Gromov--Witten invariants with target class represented by curves in the fibre of the bundle map.
This generalises some early results of Qin and Ruan~\cite{Qin--Ruan}.
We then impose an additional assumption on our normalised bundle ({\bfseries{Assumption~2}}): that $r+1+c_1>0$, where $c_1(E)=c_1p$.
This assumption is essentially an assumption on the rank of the bundle, and holds automatically in the case $r>n$. 
Assumption~2 insures that quantum multiplication of two classes $\alpha,\beta \in H^\bullet(X;\z)$ with $\deg(\alpha)+ \deg(\beta) \leq n$ has no quantum corrections determined by Gromow--Witten invariants with class $aA_1+bA_2$, with $a$ and $b$ both non-zero.

We show ({\bfseries Theorem \ref{thm:reconstruction}}) that for a Fano bundle $E$ satisfying Assumptions~1~\&~2 the small quantum cohomology of $X$ is determined by the Gromov--Witten invariants with class $A\in H_2(X)$ of degree $-K\cdot A\leq n+1$. 
This gives us good control over the invariants required to determine the quantum cohomology.
To place such a condition seems reasonable, as evidence suggests that low rank Fano bundles are split \cite{FanoBundlesSplit}. 
The projectivisation of a split bundle on $\PP^n$ is toric, so the quantum cohomology is determined by existing theory of Givental \cite{Givental_toric}.

In the second part of the paper we provide an extended example, using Theorem~\ref{thm:reconstruction} to calculate the quantum cohomology for the Fano {9-fold} $\PP(\wedge^2\Omega_{\PP^4})$, the projectivisation of second wedge of the cotangent bundle on $\PP^4$.
From the quantum multiplication data we produce the small $J$-function, following Guest~\cite{Guest:intro1, Guest:intro2}. We observe in the $J$-function some tantalising traces of modularity, in the form of the Ap\'ery numbers.

We apply the quantum Lefschetz theorem \cite{Coates--Givental} to the result for the \hbox{9-fold} to compute the quantum period for the \hbox{rank-2} Fano \hbox{3-fold} \hbox{No.~17} in the Mori--Mukai list \cite{Mori--Mukai}, which can be given as a complete intersection in the total space of the projectivised bundle.
This result is important in the classification of Fano \hbox{3-folds} as carried out in the Fanosearch programme~\cite{E6}.

While preparing the first draft of this paper the method we describe here was the only way to obtain the quantum cohomology of this Fano 3-fold.  More recently Coates--Corti\footnote{Being scooped by your own supervisors is an awkward experience.}--Galkin--Kasprzyk have shown that the quantum period of this Fano 3-fold can be computed using Abelian/Non-Abelian correspondence~\cite{Abelian_NonAbelian}.  Nonetheless the method presented here is less taxing than the alternative approach, and the methods presented in this paper also apply to more general Fano bundles.
\subsection*{Acknowledgments}
This paper is the outcome of work conducted for my PhD thesis at Imperial College London, funded by EPSRC. I would like to thank my supervisors, Alessio Corti and Tom Coates, for their advice and support.
I would also like to thank J\'anos Koll\'ar for pointing out examples of smooth projective varieties $M$, on which there exist classes $A\in H_2(M;\z))$ where, for $m>1$, $mA$ is the class of a effective curve, while $A$ is not.
\section{Notation and Assumptions}\label{sec:note}
Throughout we will use Fulton's convention for projectivised bundles; we regard $\PP(E)$ as  the bundle of one dimensional linear subspaces in the fibres of $E$, not one dimensional quotients.
Since we work exclusively with cohomology in even degree, we say that $\alpha$ has degree $\frac{n}{2}$ if $\alpha \in H^n$.

\subsection{Fano Bundles and Extremal Rays}
A vector bundle $E$ is defined to be Fano if its projectivisation, $\PP(E)$, is a Fano manifold.
Let $E\rightarrow\PP^n$ be a Fano bundle of rank $r$. Let $p=c_1(\OO_{\PP^n}(1))$ be the hyperplane class on $\PP^n$, $\xi_E=c_1(\OO_{\PP(E)}(1))$ the relative hyperplane class on $\PP(E)$. We shall write $c_i(E)=c_ip^i$ for the Chern classes of $E$.

Let $X=\PP(E)$, we have that the anti-canonical class is given by
$$-K_X=(n+1)p+c_1(E)+r\xi_E$$
Note that while it might appear that $-K_X$ depends on the normalisation of $E$, one can easily check that for $E'=E\otimes \OO (d)$, $c_1(E')=c_1(E)+dp$, while $\xi_E'=\xi_{E}-dp$.

We fix a basis for $H^\bullet(X;\z)$, $\phi_1, \dots, \phi_{r(n+1)}$, generated as a ring by $\mathbf{1}$, $p$ and $\xi_E$.
This basis is given lexicographic ordering where we take $p$ before $\xi_E$.
Let $(.,.)$ denote the intersection pairing on cohomology, so that $(\alpha,\beta)=\int_{M}{\alpha\cup\beta}$.
We denote by $\widehat{\alpha}$ the dual to $\alpha$ defined by this paring and denote by $\phi^i=\widehat{\phi_i}$ the basis dual to $\phi_i$.
We will regularly denote the cup product of cohomology classes implicitly.

As $X$ is the projectivisation of a vector bundle bundle over projective space, the Mori cone of $X$, $\overline{NE}_X$, has exactly two extremal rays $R_1,R_2$ (pg. 25 in \cite{Clemens--Kollar--Mori}). 
We may take primitive generators of the rays  $A_1, A_2 \in H_2(X;\z)$ such that $R_1=\RR_+A_1$ and $R_2=\RR_+A_2$. 
Up to a choice of ordering $A_2$ is represented by a line in the fibre of $P$ : $A_2=\PD(\widehat{\xi_E})$. 
The form that $A_1$ takes in general is less clear and, as we noted in the introduction, it may not be representable as an algebraic curve. We shall make the following assumption on the generators of the extremal rays of $X$:

\begin{ass}
We assume the following equivalent conditions:
\begin{itemize}
\item $A_1,A_2\in H_2(X;\z)$ form an integral basis.
\item $p\cdot A_1=1$ 
\end{itemize}
\end{ass}

Assumption~1 implies that $A_1=\PD\left[p^{n-1}\xi_E^{r-1}+(k+c_1)p^n\xi_E^{r-2}\right]$, for some integer $k$.
A sufficient condition for Assumption 1 to hold is if the class $A_1$ is realised by rational curves corresponding to lines in the base $\PP^n$.
This occurs when $A_1$ is the class of a section of $P|_{\ell}$ for $\ell\subset \PP^n$ a line. 
Note that $\ell$ need not be a generic line: unless the bundle is uniform it is reasonable to expect that any line realising $A_1$ will be some jumping line for the vector bundle $E$~\cite{Ancona--Maggesi}.
In greater generality one might expect that $A_1$ would be the class of a section of $P|_{C}$, for some curve $C$ of degree $> 1$ i.e. corresponding to a jumping conic or higher degree rational curve. This would contravene Assumption~1. At present we do not have examples of such a situation, though we have no general statement to rule it out.

Given Assumption 1 it trivially follows that:
\begin{align*}
p\cdot A_1&=1, &p \cdot A_2=0\\
\xi_E \cdot A_1&=k, &\xi_E \cdot A_2=1
\end{align*}

By choosing the basis $p, (\xi-kp)$, this product becomes diagonal.  This is equivalent to twisting  $E$ by $\OO(k)$, which does not change $\PP(E)$.  From here on in we assume that $E$ is normalised in this way so that the product is diagonalised, that is:
\begin{align*}
p \cdot A_1 &=1, &p \cdot A_2=0\\
\xi_E \cdot A_1&=0, &\xi_E \cdot A_2=1
\end{align*}

Under these conditions we are therefore free to assume that $E$ is normalised such that $A_1=\PD(\widehat{p})$ and $A_2=\PD(\widehat{\xi_E})$, and we do so in what follows, dropping the subscript $E$ on $\xi$.
Note that this normalisation implies that $\xi$ is nef while $\xi-p$ is not.

Since $E$ is a Fano bundle, $-K_X=(n+1+c_1)p+r\xi$ is ample. We therefore have
\begin{equation}\label{eqn:positiveness}
-K_x\cdot A_1=n+1+c_1>0. \end{equation}
With this normalisation fixed, we will later also make the assumption that:

\begin{ass}
$E$ satisfies $r+1+c_1>0$.
\end{ass}
This assumption is essentially a condition on the rank of the bundle and is fulfilled if $r>n$ (combining this with \eqref{eqn:positiveness}, Assumption 2 follows easily).
The purpose of this assumption will be made more transparent in the following section as it pertains to the degree of certain quantum corrections.
 
We shall denote by $P$ and $\Xi$ the extremal contractions given by the linear systems induced by $p$ and $\xi$ respectively. 
\[
\def\objectstyle{\displaystyle}
\xymatrix{
  X \ar[r]^{\Xi\quad}\ar[d]_P  & Y \\
  \PP^n &  }
\]

\subsection{Gromov--Witten and Quantum Cohomology}
We briefly review some aspects of Gromov--Witten theory to fix notation. 
In the following we use the fact that $A_1,A_2$ form an integer basis for $H_2(X;\z)$.

Let $X_{0,n,aA_1+bA_2}$ denote the moduli space of n-pointed genus-zero stable maps to $X$ of image class $aA_1+bA_2$ from curves of genus $0$ with $n$ marked points \cite{Kontsevich--Manin,Fulton--Pandharipande}. 
There are evaluation maps $\ev_i:X_{0,n,aA_1+bA_2}\rightarrow X, 1\leq i\leq n$, and, given cohomology classes $\alpha_1,\alpha_n\in H^\bullet(X)$, we define genus zero Gromov--Witten invariants as

$$\langle\alpha_1,..,\alpha_n\rangle_{0,n,aA_1+bA_2}^X:=\int_{[X_{0,n,aA_1+bA_2}]^{vir}}{\ev_1^*\alpha_1\wedge..\wedge \ev_n^*\alpha_n}$$
where $[X_{0,n,aA_1+bA_2}]^{vir}$ is the virtual fundamental class of $X_{0,n,aA_1+bA_2}$\cite{Li--Tian, Behrend--Fantechi, Behrend}. Recall that the virtual dimension of $X_{0,n,aA_1+bA_2}$ is \[\text{dim}_{\text{vir}}(X_{0,n,aA_1+bA_2})=\text{dim}(X)-3+n-K\cdot(aA_1+bA_2).\]

The small quantum product on a manifold $M$ is a deformation of the usual (classical) product in cohomology.
The deformation takes the form of quantum corrections, which are governed by genus-zero Gromov--Witten invariants. 

Let $\Lambda:=\C [H^2(X;\z)]$, an element of which is a finite sum $\sum_{d\in H^2(X;\z)}{\lambda_d q^d}$, with symbols multiplying as follows, ${q^{d_1} q^{d_2}=q^{d_1+d_2}}$. 
We let $q_1=q^{A_1}$ and $q_2=q^{A_2}$ be the elements of $\Lambda$ associated to $A_1$ and $A_2$ corrections respectively.

We consider the small quantum product as the following operation:
$$\star:H^\bullet(X;\C) \times H^\bullet(X;\C)\rightarrow H^\bullet(X;\C)\otimes \Lambda$$
with
$$\alpha\star \beta=\sum_{i=1}^{(n+1)r}{\sum_{a,b\geq0}{\langle\alpha,\beta,\phi_i\rangle_{0,3,aA_1+bA_2}^X\phi^i q_1^aq_2^b}}$$

We define $d_1=\mathrm{deg}(q_1)=-K\cdot A_1=n+1+c_1(E)$ and $d_2=\mathrm{deg}(q_2)=-K\cdot A_2=r$.  
With this definition, quantum product makes $H^\bullet(X,\Lambda)$ into a graded ring.
This fact, which is a simple consequence of the virtual dimension of $X_{0,n,aA_1+bA_2}$,  constrains the degree of image class for the Gromov--Witten invariants associated to quantum corrections based on the degree of the classes being multiplied.
We will often refer to the degree preserving nature of quantum multiplication by little more than ``for degree reasons". 
To calculate small quantum multiplication by $p$ and $\xi$ we need only consider 2-point invariants, due to the divisor axiom \cite{Behrend}. 

Note that Assumption 2 implies that $d_1+d_2>n$, so the quantum multiplication of two classes $\alpha,\beta \in H^\bullet(X;\z)$ with $\deg(\alpha)+ \deg(\beta) \leq n$ has no quantum corrections determined by Gromow--Witten invariants with class $aA_1+bA_2$, with $a$ and $b$ both non-zero.

\section {Reconstruction Theorem for Fano Bundles}

\begin{lem}\label{lem:zeroone}
Let $X=\PP(E)$  for $E\rightarrow Y$, a Fano bundle (not necessarily on $\PP^n$) of rank $r$ and let $Y$ be smooth of dimension $n$. Let $\xi$ be the relative hyperplane class on $X$, $\pi:X\rightarrow Y$ the induced bundle map, and $A=\PD(\widehat{\xi})$ the extremal curve corresponding to lines in the fibre of $\pi$. The Gromov--Witten invariants of type $\langle\alpha,\beta\rangle^X_{0,2,kA}$ vanish for $k\geq2$. Furthermore $\langle\alpha,\beta\rangle^X_{0,2,A}=\pi_\ast\alpha\cdot\pi_\ast\beta$ \end{lem}
\begin{proof}
This result generalises Lemmas~3.6 and 3.7 of \cite{Qin--Ruan}.

We demonstrate that for $k\geq 2$ there do not exist any rational curves of class $kA$ which satisfy the intersection properties given by $\alpha$ and $\beta$. 
Connected curves of the class $kA$ are restricted to live in a single fibre of $\pi$. 
Since we require that the curve intersects $\PD(\alpha)$ and $\PD(\beta)$, such a curve may only exist if $\PD(\alpha)$ and $\PD(\beta)$ both intersect a common fibre of $\pi$. 
We provide a dimension counting argument to show that no such common fibre exists.

The virtual dimension of $X_{0,2,kA}$ is 
$$\mathrm{dim}_vir(X_{0,2,kA})=\mathrm{dim}(X)-3 +2 +k\deg(q^A)=n+r-2+kr$$
If $k\geq 2$ then, for degree reasons, we only get non-zero invariants if $\deg(\alpha) +\deg(\beta)= n+r-2+kr\geq n+r-2+2r$. 

Letting $s$ and $t$ be the dimensions of $\PD(\alpha)$ and $\PD(\beta)$ respectively, we have that $s+t\leq n+r -2r< n$. 
This implies that $\PD(\alpha)$ and $\PD(\beta)$ do not intersect in a common fibre of $\pi$, in particular $\pi_\ast \PD(\alpha)$ and $\pi_\ast \PD(\beta)$ do not intersect generically in $Y$. 
Since there are no rational curves with the correct intersection properties, the associated Gromov--Witten invariants vanishes.

The Gromov--Witten invariants $\langle\alpha,\beta\rangle_{0,2,A}^X$ count genuine lines in the fibres of $\pi$ which intersect sufficiently generic cycles representing the classes $\PD(\alpha)$ and $\PD(\beta)$.
The fibres are projective space and it is well known that the only non-zero, two-point invariants for projective space encode the fact that there is exactly one line between two points, see e.g.\cite{Kock--Vainsencher}.
So we can rephrase the question: we wish to count fibres of $P$ which contain a point of the cycle representing $\PD(\alpha)$ and a point of the cycle representing $\PD(\beta)$. 
This can be calculated via ordinary intersection product after pushdown by $P$, as stated.

\end{proof}

Note that if $\alpha,\beta$ are elements of the basis $\{\phi_i\}$
\[
\langle\alpha,\beta\rangle_{0,2,A_2} = \left\{
        \begin{array}{ll}
            1 & \quad \text{if } \alpha=p^{n-l}\xi^{r-1},\:\beta=p^n\xi^{r-1},\: 1\leq l\leq n\\
            0 & \quad \text{otherwise}
        \end{array}
    \right.
\]

We prove two reconstruction lemmas, which together allow us to produce the full quantum cohomology for $X$ from a small number of input Gromov--Witten invariants. 
Note that throughout we are heavily reliant upon the assumption that the product between divisor classes and extremal curve classes is diagonal.
 The `$\xi$-Lemma' tells us that if we have the quantum multiplication of a class by $p$ then we know for free the multiplication of the same class by $\xi$.
 The `$p$-Lemma' says that if we know $\xi \star p^k\xi^l$, then we get $p \star p^{k-1}\xi^{l+1}$.
  Note that once we have determined quantum multiplication by divisor classes we have determined the entirety of quantum cohomology, so we restrict our interest to quantum multiplication by divisors. 
Since quantum multiplication is distributive, we only need consider multiplication of basis elements.

\begin{lem}[$\xi$-Lemma]\label{lem:xi}

Assume that for some $i,k$, with $i\geq k$, $p\star p^k\xi^{i-k}$ is known. 
We can then calculate $\xi\star p^k\xi^{i-k}$ with no geometric (i.e. moduli space) calculation. 
Put differently we already have sufficient Gromov--Witten invariants to calculate the quantum corrections in  $\xi\star p^k\xi^{i-k}$.
\begin{proof}
For degree reasons the quantum corrections in $p\star p^k\xi^{d-k}$ are determined by (and hence determine) all Gromov--Witten invariants of the form;

$$ \langle\alpha,\beta\rangle_{0,2,aA_1+bA_2}^X$$
with $ad_1+bd_2\leq i+1$ and $a\geq 1$.

Likewise, the quantum corrections for $\xi\star p^k\xi^{i-k}$ are determined by invariants of the form
$$ \langle\alpha,\beta\rangle_{0,2,aA_1+bA_2}^X$$
with $ad_1+bd_2\leq i+1$ and $b\geq 1$.

It is clear that the only Gromov--Witten invariants required to determine the quantum corrections in $\xi\star p^k\xi^{i-k}$, which are not already determined by $p\star p^k\xi^{i-k}$, are $ \langle\alpha,\beta\rangle_{0,2,A_2}^X$, which are given by Lemma~\ref{lem:zeroone}

\end{proof}

\end{lem}

\begin{lem}[$p$-Lemma]\label{lem:p}

Assume all quantum multiplication by divisor classes of classes of degree $\leq i-1$ is known. 
Further, assume that $\xi \star p^{k+1}\xi^{i-k-1}$, $i-k-1\geq 1$ is known. 
Then we can calculate $p \star p^{k}\xi^{i-k}$ with no geometric calculation.
Note that for $i\geq n+1$ and $k=n$ , $\xi \star p^{n+1}\xi^{i-n-1}=0$, since $p^{n+1}=0$
\end{lem}
\begin{proof}
Since quantum cohomology is both associative and commutative, we can use the knowledge of multiplication in lower degree to make the following manipulations:

\begin{align*}
p\star p^k \xi^{i-k}=& p\star( \xi\star p^k\xi^{i-k-1}-f_2q_2)\\
 &=\xi \star p^{k+1} \xi^{i-k-1}+\xi \star f_1q_1 - p\star f_2q_2\\
 &= p^{k+1}\xi^{i-k} +\text {known quantum corrections}
\end{align*}
where $f_1$ is the quantum correction from $p\star p^k\xi^{i-k-1}$ and  $f_2$ the correction from $\xi\star p^k\xi^{i-k-1}$. 
Note that these are of degree $i-d_1$ and $i-d_2$ respectively so their multiplication by divisors is, by assumption, known.
At each step the quantum corrections are all known and governed by invariants we already have, since they are necessarily of lower degree.
\end{proof}

For degree reasons quantum multiplication, by divisor classes, of classes of degree $i\leq n$ is completely determined by the Gromov--Witten invariants of the form $\langle\alpha,\beta\rangle^X_{0,2,A_2}$ and $\langle\alpha,\beta\rangle^X_{0,2, kA_1}$ for $kd_1\leq n+1$. 
Note in particular that there can be no corrections coming from $\langle\alpha,\beta\rangle^X_{0,2, A_1+A_2}$ invariants. 
Every basis element of degree $i\geq n+1$ is divisible by $\xi$, so we can apply the $p$-Lemma, the proof of which requires that we can divide out a factor of $\xi$. 

In the process of producing an algorithm from Lemmas~\ref{lem:xi} and \ref{lem:p} to carry out the reconstruction process, we consider cohomology classes as vectors in the lexicographical basis $\phi_i$. Quantum multiplication by the basis classes $p$ and $\xi$ can be considered as left multiplication of cohomology vectors by $r(n+1)\times r(n+1)$ matrices, $M_p$ and $M_\xi$ respectively. 

With this view point in mind, we see that the preceding lemmas can be reinterpreted. 
\begin{lem}[$p$-lemma as linear algebra]
Given the $i^{\text{th}}$ column of $M_p$ we can determine the $i^{\text{th}}$ column of $M_\xi$.
\end{lem}

\begin{lem}[$\xi$-lemma as linear algebra]
Assume the first $i$ columns of $M_p$ and $M_q$ have been determined. We can calculate the $i+1^{\text{st}}$ column of $M_p$ using linear algebra.
\end{lem}

\begin{theorem}\label{thm:reconstruction}
Let $X=\PP(E)$ with $E\rightarrow\PP^n$ a Fano bundle of rank $r$, such that $r+1+c_1(E)>0$. The quantum cohomology of $X$ can be reconstructed from the Gromov--Witten invariants which map to target class $A$ such that $-K\cdot A\leq n+1$. In particular in our fixed basis that is invariants of the form $\langle\alpha,\beta\rangle^X_{0,2,A_2}$ (which are determined by Lemma~\ref{lem:zeroone}) and $\langle\alpha,\beta\rangle^X_{0,2,kA_1}$ for $kd_1\leq n+1$.
\end{theorem}
\begin{proof}
The proof proceeds by the construction of an algorithm. We alternately apply Lemma~\ref{lem:p} followed by Lemma~\ref{lem:xi} to calculate the multiplication for all basis elements following the lex ordering. We produce this algorithm in pseudo-code below.

We first describe in words, roughly how the algorithm proceeds.
Consider the first basis element  of degree $n+1$ (under lexicographical ordering), i.e. $p^n\xi$. 
Since we are in the special case of Lemma~\ref{lem:p} ($p^{n+1}=0$) we obtain $p\star p^n\xi$. We are now in a position to apply Lemma~\ref{lem:xi}, to obtain $\xi \star p^n\xi$.
One can easily check that we can now apply both lemmas to the next basis element in degree $n+1$ and so on. After obtaining the multiplication data for classes of degree $n+1$ we repeat the process for degree $n+2$ (again obtaining the first multiplication for free due to the vanishing of $p^{n+1}$. We repeat until we obtain the full quantum multiplication data.

The input for the reconstruction process is a pair of $r(n+1)\times r(n+1)$ matrices $M_p$ and $M_\xi$. The columns corresponding to multiplication of degrees $\leq n$ are known\footnote{in the case $r>n$ this is the first $\frac{1}{2}(n+2)(n+1)$ columns} and we initialise the unknown entries as zero. By convention we label arrays with the first entry given index 1. We may think of the reconstruction process as giving an algorithm to fill in the rest of the matrices $M_p$ and $M_\xi$

The result of Lemma~\ref{lem:p} can usefully be written as
$$p\star p^{k}\xi^{i-k}=\xi\star p \star p^{k}\xi^{i-k-1}-p\star(\xi \star p^{k}\xi^{i-k-1}-p^{k}\xi^{i-k})$$
We use this form in the algorithm as it lends itself easily to calculation by matrix multiplication.

We define the following functions for use in our pseudo-code:
Let $\Ind{d,k}$ be the position that $p^{k}\xi^{d-k}$ appears in the basis of $H^\bullet(X)$, when given lexicographical ordering 
In the case that $r>n$, this is given by
$$
\left\{
		\arraycolsep=3pt
        \begin{array}{ll}
            \frac{1}{2}(d+2)(d+1)-k, &  \text{for } d\leq n\\
            \frac{1}{2}(n+1)(2d-n+2)-k, &  \text{for } n<d<r\\
            \frac{1}{2}((n+1)(2r-n)+(d-r+1)(2n-d+r))\;&\\
            \quad -k+d+1-r, &  \text{for } r\geq d
        \end{array}
    \right.
$$

Next, $\V{d,k}$ is the vector corresponding to the monomial $p^k\xi^{d-k}$

$$
\V{d,k}[j] = \left\{
        \begin{array}{ll}
            1 & \quad \text{for } j=\Ind{d,k}\\
            0 & \quad \text{otherwise } 
        \end{array}
    \right.
$$

Finally we define the Term function, which extracts from a polynomial the term specified along with the corresponding coefficient. It is given by the expression $\term{\text{polynomial},\text{monomial}}$.
e.g. $\term{5x^2+3xy+2x+1,x}=2x$

Let $C$ be an $r$-vector with $C[i]:=-c_i(E)$, note that we set the entries to $-c_i(E)$ to determine the form $\xi^r$ takes.

We produce the following algorithm. Note that we aim for clarity of exposition and as such the algorithm is not minimised in terms of calculational expense. 
\begin{algorithmic}
\State \#\# d loops over all the degrees to calculate
\For{$d=n+1 \to r+n$}
	\State \#\# We loop over all classes in degree d
	\State \#\# Note that the expression below in I(.,.) determines the number of 
	\State\#\# basis elements of the given degree
	\State\#\# together d and k determine which column of the matrix 
	\State \#\# we're filling in
	\For{$k=0 \to (\Ind{d+1,n}+\Ind{d,n}-1)$}
	\State\#\# looping over j goes through each row of the selected column
		\For{$j=1 \to (n+1)r$}
			\State \#\# applying the $p$-lemma
			\State \#\# note that we use the special form described above
			\State $M_p[j,\Ind{d,n-k}]:=(M_\xi \cdot M_p \cdot \V{d-1,n-k}-M_p\cdot(M_\xi \cdot\V{d-1,n-k}-\V{d,n-k}))[j]$
			\State \#\# applying the $\xi$-lemma
			\State \#\# we loop over the possible powers of $q_1$
			\For{$s=1\to (n+1)r$}
			\State \#\# Extracting the quantum corrections from the relevant 			
			\State \#\# $p$ multiplication. We just pull out the data
			\State \#\#  from the corresponding column in $M_p$
			\State\#\# Note the important factor $\frac{1}{s}$
			\State  \#\# from the divisor axiom
				\State $M_\xi[j,\Ind{d,n-k}]:=M_\xi[j,\Ind{d,n-k}]+\frac{1}{s}\term{M_p[j,\Ind{d,n-k}],q_1^sq_2}$
			\EndFor
		\EndFor
		\State \#\# Testing for special case where class $=p^a\xi^{r-1}$
		\State \#\# i.e the case where we need to insert chern numbers for $\xi^r$.
	\If{$k=n+r-d-1 \textbf{ and } d\geq r-1$}
		\State \#\# Insert additional quantum correction in special case
		\State \#\# This Gromov--Witten invariants is determined by Lemma~\ref{lem:zeroone}
		\State $M_\xi[\Ind{d-r+1,n-k},\Ind{d,n-k}]:=q_2$
		\State \#\# Insert classical multiplication in special case, 
		\State \#\# i.e $\xi^r=-c_1p\xi^{r-1}-c_2p^2\xi^{r-2}-\dots$
		\State \#\# We shift the entries as we insert  them to account for the
		\State \#\#  powers of $p$ that the class we're multiplying includes. 
		\For{$s=0 \to k-1$}
			\State $M_\xi[\Ind{d+1,n-s},\Ind{d,n-k}]:=C[k-s]$
		\EndFor
	\Else
		\State \#\# Insert classical multiplication in generic case
		\State \#\# i.e. dealing with all cases where we don't get $\xi^r$
		\State \#\# note that by Lemma~\ref{lem:zeroone} there are no quantum corrections 
		\State\#\# in this case
		\State $M_\xi[\Ind{d+1,n-k},\Ind{d,n-k}]:=1$
	\EndIf
	\EndFor
\EndFor
\end{algorithmic}
\end{proof}

\begin{rem}
The  reconstruction algorithm proposed in the proof to Theorem \ref{thm:reconstruction} makes use of the fact that, as a special case of Lemma~\ref{lem:p}, $p\star p^n\xi^{i-n}$ is known for $i>n$. 
The algorithm takes as input the columns of the matrices for quantum multiplication by $p$ and $\xi$ which correspond to classes of degree $i\leq n$.

The combination of Lemmas \ref{lem:xi} and \ref{lem:p} in fact shows additionally that in any degree~$i\leq n$, given the input $p\star p^i$ (and the Gromov--Witten invariants determined by Lemma~\ref{lem:zeroone}), we can determine the quantum multiplication of all other classes of degree $i$ without explicit calculation of any other Gromov--Witten invariants. This remark shows that the quantum cohomology for $X$ can be recovered from  $p\star p^k$, for $1<k\leq n-1$ alone.
\end{rem}

\subsection{Special Cases}

There are two special cases in which Theorem~\ref{thm:reconstruction} can be significantly strengthened. 
The first is when the second extremal contraction realises $X$ as the projectivisation of a vector bundle on some other space, the second when the contraction realises $X$ as the blowup of some smooth space in a smooth locus. 
In both cases the entire quantum cohomology is determined by counting lines in the fibres of the two extremal maps $P$ and $\Xi$.
We first prove a lemma regarding the Gromov--Witten invariants of a blow-up of a smooth sub-variety in a smooth ambient space.

\begin{lem}\label{lem:blowup} 
Let $X:=\PP(E)$ with $E$ a Fano bundle satisfying Assumption 1. Let the extremal contraction associated to $A_1$ be given by $\Xi:X\rightarrow Y$ is the blow up of  $Z\subset Y$, with both $Y$ and $Z$ smooth. 
Then the Gromov--Witten invariants of the form $\langle\alpha,\beta\rangle^X_{0,2,kA_1}$ vanish for $k\geq2$.
 
Furthermore, $\langle\alpha,\beta\rangle^X_{0,2,A_1}=\Xi|_\ast\iota^\ast\alpha\cdot\Xi|_\ast\iota^\ast\beta$, where $\Xi|$ is the restriction of $\Xi$ to $D\subset X$, the exceptional divisor of the blow up, and $\iota: D\hookrightarrow X$ is the embedding of $D$ in $X$.
\end{lem}
\begin{proof}

We summarise the geometry of the statement in the following diagram

\[
\def\objectstyle{\displaystyle}
\xymatrix{
  D\ar@{^{(}->}[r]^\iota \ar[d]^{\Xi|}  & X\ar[d]^\Xi \\
  Z\ar@{^{(}->}[r]& Y}
\]
Curves of class $kA_1$ are restricted to the fibres of $\Xi$. 
It is clear that they must live in the exceptional fibres of $\Xi$, which are precisely the fibres of $\Xi|$. The calculation may be carried out inside $D$:
stable maps with target $X$ and class $kA_1$ are in one to one correspondence with stable maps with target $D$ and class $\iota^\ast kA_1$.

Since $Y$ and $Z$  are both smooth, the exceptional locus, $D$, is given by $\PP(N_{Z/Y})$, when $N_{Z/Y}$ is the normal bundle of $Z$ in $Y$. We are now in the case of Lemma~\ref{lem:zeroone}.
\end{proof}

We have proved the following corollary to Theorem~\ref{thm:reconstruction}
\begin{cor}\label{cor:specialCase}
Let $X=\PP(E)$ with $E\rightarrow\PP^n$  Fano bundle satisfying Assumptions 1 \& 2 and with the second extremal contraction $\Xi:\PP(E)\rightarrow Y$ given by either
\begin{enumerate}
\item the projectivisation of a bundle $E'\rightarrow Y$
\item the blow up of a smooth subvariety $Z$ inside smooth $Y$
\end{enumerate}
then the quantum cohomology of $X$ can be reconstructed from the Gromov--Witten invariants of the form  $\langle\alpha,\beta\rangle_{0,2,A_1}$ and $\langle\alpha,\beta\rangle_{0,2,A_2}$. These Gromov--Witten invariants are determined by Lemmas \ref{lem:zeroone} and \ref{lem:blowup}.
\end{cor}

\section{The Geometry of Rank-2 Fano 3-fold No. 17}
Our aim is to compute the quantum period sequence of $M$, the rank-2 Fano 3-fold No. 17 in Mori--Mukai \cite{Mori--Mukai}. $M$ can be embedded as a complete intersection in $X=\PP(E)$ for some Fano bundle $E$ which we now describe.

Fix a vector space $V\simeq \C^5$ with $\PP^4:=\PP(V)$.
Let $E:=\Omega^2_{\PP(V)}(2)$ be the second wedge of the bundle of holomorphic differentials on $\PP^4$, twisted by $\OO(2)$.
Let $X$ be the total space $\PP(E)$. We can be naturally view $X$ as the blow-up of the Pl\"ucker embedding of the Grassmannian $G(2,V^\star)\subset\PP(\wedge^2V^\star)$.  

As above, let $P$ be the canonical map induced by $X$'s bundle structure and $\Xi$ the contraction induced by the second extremal ray. 
We illustrate this in the following diagram.
\[
\def\objectstyle{\displaystyle}
\xymatrix{
  X \ar[r]^{\Xi\quad}\ar[d]_P  &\PP(\wedge^2 V^\star) \\
  \PP(V) &  }
\]

\subsection{Cohomology and Extremal Rays of $X$}
The ordinary cohomology of $X$ is determined by the Chern classes of $E$, which are easily obtained from the following exact sequence of bundles, taking the second wedge of the Euler sequence on $\PP(V)$.
By abuse of notation we refer by $V^\star\otimes\OO$ to the trivial $V^\star$ bundle on $\PP(V)$.

\begin{equation}\label{eq:defseq}0\rightarrow E\rightarrow \wedge^2 V^\star\otimes \OO \xrightarrow{v\lrcorner} V^\star \otimes \OO(1) \rightarrow \OO(2) \rightarrow 0\end{equation}

Here $v\lrcorner$ denotes the contraction by the vector $v\in V$ representing the point in $\PP(V)$. From \eqref{eq:defseq} it is clear that $X$ embeds into $\PP(V)\times \PP(\wedge^2 V^\star)\cong\PP^4\times\PP^9$.
The total Chern class is $c(E)=1-3p +5p^2 - 5p^3$ and hence

 \begin{equation}\label{eq:cohom}
 H^\bullet(X)=\frac{\C[p,\xi]}{(p^5,\xi^6-3p\xi^5 +5p^2\xi^4 - 5p^3\xi^3)}
 \end{equation}

We check that $\xi$ is nef and prove that $E$ is a Fano bundle.
\begin{lem}\label{lem:nefness}
$\xi$ is nef.
\begin{proof}
Since $\xi$ is the relative hyperplane on $\PP(E)$, it is a quotient of $P^* E^\star$. $E^\star$ is generated by global sections (one sees this by dualising \eqref{eq:defseq}) so $\xi$ is also generated by global sections and is therefore nef.
\end{proof}
\end{lem}

Since $-K_X=6\xi+2p$, and both $p$ and $\xi$ are nef we see that $-K_X$ is ample; $X$ is Fano.

In $H_2(X,\z)$ we fix classes $A_1,A_2$. Recall that $A_2$ is the class of a line in the fibre $P:X\rightarrow\PP(V)$.  
In this example $A_1$ is the class of a line in an exceptional fibre of $\Xi:X\rightarrow \PP(\wedge^2V^\star)$ and is isomorphic, by $P$, to a generic line in $\PP^n$. We have
$$p\cdot A_1=0,  \xi\cdot A_1=1, p\cdot A_2=1, \xi\cdot A_2=0$$
from which follows $A_1=\PD(p^3 \xi^5-3p^4 \xi^4)$ and $A_2=\PD(p^4 \xi^4)$.

\begin{lem}
$A_1,A_2$ generate the extremal rays of the Mori cone.
\end{lem}
\begin{proof}
Both $p$ and $\xi$ are nef, so, given the homology class $aA_1+bA_2$ of any curve T, we have $a=\xi\cdot T \geq 0$ and $b=p\cdot T\geq 0$. 
On the other hand both $A_1$ and $A_2$ are contained in the Mori cone and therefore form the extremal rays.	
\end{proof}

The following discussion shows that $X$ is given by the blow up of $\PP(\wedge^2V^\star)$ along the Pl\"ucker embedding of $G(2,V^\star)$.

A point $[w] \in \PP(\wedge^2V^\star)$ is represented by a 2-form $w$.
The fibre $\Xi^{-1}([w])$ over this point consists of the subset of points of $\PP(V)$ which are represented by a vector $v$ annihilated by $w$. 
Considering $w$ as an antisymmetric $5 \times 5$ matrix $A\colon V \to V^\star$, it is clear that  $w$ is generically of rank 4, i.e. there is a one-dimensional space of vectors $v$ annihilated by $w$. So, generically the fibre $\Xi^{-1}([w])$ is a point. 
When $w$ drops ranks it must be rank 2, since we throw away the zero form upon projectivisation. Where $w$ is of rank 2, the fibre is a $\PP^2$.  
The locus of rank-2 forms corresponds to those elements of $\wedge^2V^\star$ which are decomposable to $w_1\wedge w_2$, where $w_1, w_2$ are 1-forms.  
This is exactly the Pl\"ucker embedding of $G(2,V^\star)\hookrightarrow\PP(\wedge^2V^\star)$ as a codimension-3 subvariety. 
Hence $X$ is isomorphic to the blow up of $G(2,V^\star)\subset \PP(\wedge^2V^\star)$.

\begin{lem}\label{lem:pf}
 The rational map $\PP(\wedge^2 V^\star) \dasharrow \PP(V)$ is given by the linear system of quadrics containing $G(2,V^\star)$.
\end{lem}

\begin{proof}
 The map sends a 2-form $w$, thought of as an antisymmetric $5\times 5$ matrix $A\colon V \to V^\star$, to its kernel. 
By a version of the Cramer rule, we can describe the map explicitly by sending $A$ to the vector of $4\times 4$ Pfaffians:
 \[
 \pf(A)=(\pf_0(A),\dots,\pf_4(A)).
 \]
 The statement then reduces to the fact that these Pfaffians generate the ideal of $G(2,V)\subset \PP(\wedge^2 V^\star)$. 
\end{proof}

\subsection{$M$ as a complete intersection in $\PP(E)$}\label{sec:intersection}
We denote by $M$ the Fano \mbox{3-fold} No.~17 in the Mori--Mukai list of rank~2 Fano \mbox{3-folds}\cite{Mori--Mukai}. 
According to Mori--Mukai, $M$ is the blow-up of a 3-dimensional quadric $Q\subset \PP^4$ with centre $\Gamma \subset Q$, a nonsingular curve of genus $1$ and degree $5$.

\begin{lem}
$M$ is a complete intersection of type $p\cap \xi^5$ in $\PP(E)$
\end{lem}
\begin{proof}

It is well--known that the Plucker embedding $G(2, V^\star)\hookrightarrow \PP(\wedge^2 V^\star)$ is of degree five.  Using adjunction, one can easily check that the curve given by the complete intersection of 5 general hyperplanes with the Grassmannian has trivial canonical bundle and hence is genus 1:
\[
\Gamma = h_1 \cap \cdots \cap h_5 \cap G(2, V^\star) \subset h_1 \cap \cdots \cap h_5 \cong \PP^4,
\]  

all of this taking place in the natural ambient $\PP(\wedge^2 V^\star)$. Since $\Xi \colon X \to  \PP(\wedge^2 V^\star)$ is the blow-up of $\PP (\wedge^2 V^\star)$ along the Pl\"ucker embedding of $G(2,V^\star)$, the discussion makes it clear that $M$ is the complete intersection
\[
M=\Xi^\ast(h_1) \cap \cdots \Xi^\ast (h_5) \cap \widetilde{Q} \subset X
\]
where $\widetilde{Q}$ is the proper transform of a quadric containing $G(2,V^\star)$, i.e., by Lemma~\ref{lem:pf}, a section of $p$.
\end{proof}

\begin{cor}
$-K_X=(p+\xi)|_X$
\end{cor}

\section{The Quantum Cohomology of $X$ and $M$}
In this section we use the reconstruction theorem (Theorem~\ref{thm:reconstruction}) to calculate the quantum cohomology of $X$.
By passing to a certain generating function, the $J$-function, of $X$ we use Quantum Lefschetz \cite{Coates--Givental} to obtain information about the quantum cohomology of $M$.

Since $X=\PP(E)$ is the projectivisation of a Fano bundle and the extremal contraction $\Xi$ is the blow up of $G(2,5)\subset \PP^9$ the quantum cohomology follows from Corollary~\ref{cor:specialCase}.

We make use of the Schubert calculus for $G(2,5)$, following notational conventions from  \cite{Griffiths--Harris}.
Let $N$ be the normal bundle to the embedding of $G(2,V^\star)$ in $\PP(\wedge^2V^\star)$.
The exceptional divisor $D$ is given by the projectivisation of $N$. 
Let Q be the tautological quotient bundle on $G(2,V^\star)$.  
The normal bundle to the embedding into $\PP(\wedge^2V^\star)$ is given by $Q^\star(2\sigma_1)$ though for the sake of convenient relations in cohomology we will work instead with $Q^\star$ (of course $\PP(N)$ and $\PP(Q^\star)$ are isomorphic). 
The cohomology of this bundle is given by 

$$H^\bullet(\PP(Q^\star))=\frac{H^\bullet(G(2,V^\star))[\eta]}{(\eta^3+\sigma_1\eta^2+\sigma_2\eta+\sigma_3)}$$
where $\eta$ is the relative hyperplane class of $\PP(Q^\star)$.

The following diagram describes the geometry of the situation.

\[
\def\objectstyle{\displaystyle}
\xymatrix{
  \PP(Q^\star)\ar@{^{(}->}[r]^\iota \ar[d]^{\Xi|}  & Y\ar[d]^\Xi \\
  G(3,V)\ar@{^{(}->}[r]& \PP(\wedge^2V^\star)}
\]

For the purpose of calculation, note that $\iota^*\xi=\sigma_1$ and $\iota^* p=\eta$.
  \begin{theorem}
 $$QH^\bullet(X)=\frac{\C[p,\xi,q_1,q_2]}{(R_1,R_2)}$$
where \[R_1=p^{\star5}-q_1^2p+2q_1^2\xi+2q_1p^{\star3}-2q_1p^{\star2}\star\xi-q_1p\star\xi^{\star2}-q_1\xi^{\star3}\] and \[R_2=\xi^{\star6}-3p\star\xi^{\star5} +5p^{\star2}\star\xi^{\star4} - 5p^{\star3}\star\xi^{\star3}-q_2-5q_1p\star\xi^{\star3}+10q_1\xi^{\star4}\] 
\begin{proof} 
Quantum multiplication is determined by Corollary \ref{cor:specialCase}.
From Theorem~2.2 in \cite{Siebert--Tian}, all that then remains is to evaluate the relations of the classical cohomology \eqref{eq:cohom}, replacing classical multiplication with quantum multiplication, from which the statement follows.

\end{proof}
\end{theorem}

\subsection{Quantum differential structure}
Our ultimate goal is to compute the $J$-function \cite{Givental:Equivariant}, a generating function for certain genus-zero Gromov--Witten invariants, of $M$. 
As $M$ is a complete intersection in $X$ we will use the Quantum Lefschetz theorem
\cite{Coates--Givental}, which expresses certain genus-zero Gromov--Witten
invariants of $M$ in terms of invariants of the ambient space $X$.
The input that we need for the Quantum Lefschetz theorem is the $J$-function
of $X$ as well as the direct sum of line bundles which describe $M$ as a complete intersection.
In this section we describe a method to obtain an arbitrary number of terms in the power series expansion of the
$J$-function of $X$, by solving a system of differential equations,
called the quantum differential equations for $X$.
Here we follow
closely the excellent papers of Guest
\cite{Guest:intro1,Guest:intro2}.

Recall from Section \ref{sec:note}, $\phi_1,\dots,\phi_{30}$ is the lexicographical basis of $H^\bullet(X;\z)$ in $p$ and $\xi$, with $\phi^1,\dots,\phi^{30}$ the dual basis given by the intersection pairing. The $J$-function of $X$ is the function $H^2(X;\C)\rightarrow H^\bullet(X;\C)\otimes\C[[1/z]]$ defined by

$$J_X(t)=e^{t/z}\left(1+\sum_{\epsilon=1}^{30}{\sum_{a,b\geq 0}{e^{d\cdot t}\left<\frac{\phi^\epsilon}{z(z-\psi)} \right>^X_{0,1,aA_1+bA_2}}}\right)$$
where we expand $\left<\frac{\phi^\epsilon}{z(z-\psi)} \right>_{0,1,aA_1+bA_2}$ as $\sum_{k\geq0}{\langle\phi^\epsilon\psi^k\rangle_{0,1,aA_1+bA_2}\frac{1}{z^{k+2}}}$.

Since $t\in H^2(X,\C)$ is nilpotent as an element of $H^\bullet(X)$ the expression $e^{t/z}$ makes sense in $H^\bullet(X)\otimes\C[[1/z]]$.

By writing $t\in H^2(X;\C)$ as $t=t_0+t_1p+t_2\xi$, we regard the $J$-function as a function of $q_1=e^{t_1},q_2=e^{t_2}$:
$$J_X(q)=q_1^{p/z}q_2^{\xi/z}\left(1+\sum_{\epsilon=1}^{30}{\sum_{a,b\geq 0}q_1^a q_2^b\left<\frac{\phi^\epsilon}{z(z-\psi)} \right>_{0,1,aA_1+bA_2}}\right)$$

Here $q_1^{p/z}=\mathrm{exp}(p\, \mathrm{ log}q_1/z)=\mathrm{exp}(tp/z)$

$J_X$ satisfies a system of differential operators, called quantum differential operators \cite{Guest:intro2}. Let $M_p(q),M_\xi(q)$ denote the matrices of quantum multiplication by $p$ and $\xi$ with respect to the basis $\phi_1,\dots,\phi_{30}$, which are easily obtained from the Corollary~\ref{cor:specialCase} and the discussion above. The reader may find $M_p(q),M_\xi(q)$ in Appendix~\ref{sec:MMatrix}.

 Consider the system of differential equations:
$$z\,q_1\frac{\partial}{\partial q_1}s=M_p(q)s$$
$$z\, q_2\frac{\partial}{\partial q_2}s=M_\xi(q)s$$ 

where s is a vector valued function of $t\in H^2(X;\C)$ (or equivalently a multivalued vector function of $q_1,q_2$.)

This system admits a fundamental solution matrix, the rows of which are given by vectors $J_i$

$$S=\left(\begin{tabular}{ccccc}
-&-&$J_1$&-&-\\
-&-&$J_2$&-&-\\
&& $:$ &&\\
-&-&$J_{30}$&-&-\\
\end{tabular}\right)
$$

The row-vector $J_{30}$ is the expansion of the $J$-function as a vector valued function  in $H^\bullet(X;\C)\otimes\C[[1/z]]$ in terms of the basis $\phi_1,\dots,\phi_{30}$.
The differential system is equivalent to 60 differential equations in $J_1$ through $J_{30}$
By solving for the rows $J_1$ through $J_{29}$ in terms of $J_{30}$ we are left with $31$ differential equations in $J_X$.  By applying Groebner basis techniques we find a generating set for the ideal formed by these equations in the Weyl algebra.
 These (non-unique) differential equations are \emph{quantum differential equations} and define $J_X$ up to scalar.
  Using the RosenfeldGroebner tool in the DifferentialAlgebra package of Maple 16 we obtain the following result.

\begin{lem}
The ideal of quantum differential operators for $X$ is generated by
\begin{multline}
\Delta_1=D_{2}^{10} - q_{2} D_{1}^{4} - 2 q_{2} D_{1}^{3} D_{2} - 4 q_{2} D_{1}^{2} D_{2}^{2} - 3 q_{2} D_{1} D_{2}^{3} - q_{2} D_{2}^{4}\\
\quad - 2 q_{1} q_{2} D_{1}^{2} - 2 z q_{2} D_{1}^{3} - 2 z q_{1} q_{2} D_{1} D_{2} - 8 z q_{2} D_{1}^{2} D_{2} - 3 q_{1} q_{2} D_{2}^{2} \\ 
\quad\quad- 9 z q_{2} D_{1} D_{2}^{2} - 4 z q_{2} D_{2}^{3} - q_{1}^{2} q_{2} - 4 z q_{1} q_{2} D_{1} - 4 z^2 q_{2} D_{1}^{2} - 7 z q_{1} q_{2} D_{2} \\
- 9 z^2 q_{2} D_{1} D_{2} - 6 z^2 q_{2} D_{2}^{2} - 5 z^2 q_{1} q_{2} - 3 z^3 q_{2} D_{1} - 4 z^3 q_{2} D_{2} - z^4 q_{2}
\end{multline}
\begin{multline}
\Delta_2=D_{1} D_{2}^{7} - 2 D_{2}^{8} + 5 q_{2} D_{1}^{2} + 5 q_{2} D_{1} D_{2} + 2 q_{2} D_{2}^{2} + 5 q_{1} q_{2} + 5 z q_{2} D_{1}\\
 + 4 z q_{2} D_{2} + 2 z^2 q_{2}
\end{multline}
\begin{multline}
\Delta_3=5 D_{1}^{3} D_{2}^{3} - 5 D_{1}^{2} D_{2}^{4} + 3 D_{1} D_{2}^{5} - D_{2}^{6} + 5 q_{1} D_{1} D_{2}^{3} - 10 q_{1} D_{2}^{4} + q_{2}
\end{multline}

\begin{multline}
\Delta_4=D_{1}^{5} + 2 q_{1} D_{1}^{3} - 2 q_{1} D_{1}^{2} D_{2} - q_{1} D_{1} D_{2}^{2} - q_{1} D_{2}^{3} + q_{1}^{2} D_{1} + 2 z q_{1} D_{1}^{2}\\
 - 2 q_{1}^{2} D_{2} - 3 z q_{1} D_{1} D_{2} - 2 z q_{1} D_{2}^{2} + z^2 q_{1} D_{1} - 2 z^2 q_{1} D_{2}
\end{multline}

where $D_i=zq_i\frac{\partial}{\partial q_i}$

\end{lem}

\begin{cor}$J$ satisfies $\Delta_1J=\Delta_2J=\Delta_3J=\Delta_4J=0$ 

\end{cor}

The identity component of $J$, denoted $J^0$, is a power series in $q_1$ and $q_2$, i.e. $J^0=(J,\phi^0)=\sum_{i,j\geq0}{c_{i,j}q_1^iq_2^j}$.
The differential system gives recursion relations for the coefficients in this power series.
The coefficients are fixed by demanding that $c_{0,0}=1$.
We can find $J^0$ up to arbitrary order and by observation of finite terms try to find a general solution for the coefficients  $c_{i,j}$ which solves the differential system.
Such a solution, however, has not been forthcoming.
We present the coefficients $c_{i,j}$ for $i\leq7,j\leq7$  in the the following matrix $A=(a_{i,j})$. 
We have cleared the denominators by setting $a_{i+1,j+1}=i!^2j!^6c_{i,j}$.

$$ A:=\left[ \begin {array}{cccccccc} 1&1&1&1&1&1&1&1\\ \noalign{\medskip}0
&5&20&51&104&185&300&455\\ \noalign{\medskip}0&4&73&447&1756&5320&
13539&30373\\ \noalign{\medskip}0&0&90&1445&10904&55220&216110&703955
\\ \noalign{\medskip}0&0&36&2148&33001&282085&1690515&7926751
\\ \noalign{\medskip}0&0&0&1500&54500&819005&7606080&51405305
\\ \noalign{\medskip}0&0&0&400&50350&1447150&21460825&211463875
\\ \noalign{\medskip}0&0&0&0&24500&1590050&39750270&584307365
\end {array} \right] 
 $$

We remark that the coefficients are all zero below the leading `slant diagonal'.
Additionally, in the form presented where we have cleared denominators, the leading diagonal $a_{i,i}$ is given by the Ap\'ery numbers\cite{Aperynumbers}, so we expect that a closed formula may be hard to find.
The occurrence of the Apery numbers here may indicate hidden modular symmetries of $J_0$ (cf \cite{Beukers,Zagier}). Golyshev has observed a striking connection between the quantum differential equations for Fano 3-folds of Picard rank one and modular forms  \cite{Golyshev}, and it is possible that this connection persists to the case of higher Picard rank.

\subsection{The Regularised Quantum Period Sequence for $M$} 
As we noted in the preceding section, given a complete intersection $M\subset X$, Quantum Lefschetz \cite{Coates--Givental} allows us to calculate part of the $J_M$ from $J_X$ and the Chern classes of the line bundles defining the intersection.
In our example we are only interested in the identity component of the $J$-function, Quantum Lefschetz provides the entirety of $J_M^0$ from $J_X^0$. 
Our aim is not to explain how Quantum Lefschetz works, but how one uses it in practical examples.

With this in mind we briefly outline the process of applying Quantum Lefschetz. The following statements hold for any ambient space $X$ and complete intersection $M$.

Let $\mathcal{E}=\oplus L_i$ be the direct sum of line bundles corresponding to $M$ and $\rho_i=c_1(L_i)$ the first Chern class of the line bundle summands. Given the $J$-function of $X$, $J_X(t,z)=\sum_{d\in H_2(X)}{J_{d}(t,z)q^d}$, one forms the hypergeometric modification

$$I_{X,M}(t,z)=\sum_d{J_{d}(t,z)q^d\prod_{i}\prod_{k=1}^{\rho_i \cdot d}{(\rho_i+kz)}}$$
 
We also consider a formal function (with same domain and target) $J_{X,M}(t,z)$, defined in \cite{Coates--Givental}, which has the following property
$$e(\mathcal{E})J_{X,M}(j^\ast u,z)=j_\star J_M(u,z)$$
where $e(\mathcal{E})=\prod_i\rho_i$ is the Euler class of the bundle $\mathcal{E}$.
The relation between $J_M$ and $J_X$ is indirectly realised by the mirror map, which relates $I_{X,M}$ to $J_{X,M}$.
The mirror map is determined by comparing the asymptotics of the expressions. 
Considered as a power series in $z^{-1}$, $J_{X,M}$ is the unique function with the form $J_{X,M}= z + t + O(z^{-1})$.
We may write $I_{X,M}$ in the form $F(t)z+G^0(q,t)\phi_0+\sum_{i=1}^{r}{G^i(t)\phi_i}+O(z^{-1})$.
By homogeneity considerations writing $I_{X,M}$ in this form is practicable, even in the case that $J_X$ is only known up to finite order in $q$, as in the case in point. 
The mirror map is given by the following;
\begin{equation}\label{eq:I=J}
\frac{I_{M,X}}{F(t)}=J_{X,M}\left({\frac{G^0(q,t)}{F(t)}\phi_0+\sum^r_{i=1}{\frac{G^i(t)}{F(t)}\phi_i},z}\right)=e^{\frac{G^0(q,t)}{F(t)}\phi_0}J_{X,M}(\tau,q)
\end{equation}
where $\tau=\sum^r_{i=1}{\frac{G^i(t)}{F(t)}\phi_i}$. The second equality follows from the string equation and the definition of $J$.

The procedure may summarised as follows:
\begin{enumerate}
 \item Calculate $J_X$
 \item Produce the hypergeometric modification $I_{X,M}$
 \item Calculate the mirror map from the asymptotics of $I_{X,M}$
 \item Produce $J_{X,M}$ from $I_{X,M}$ using the mirror map
\end{enumerate}
Note that we do not obtain the entirety of $J_M$ by comparison with $J_{X,M}$: some information is lost in the pushforward. However, we can recover the full identity component $J_M^0$ by following the described method with $J_X^0$, since pushforward of the identity \emph{is} cup with the Euler class.

We now return to our particular example and previous definitions for $M\subset X=\PP(\Omega^2(2))$ and, proceeding as above,  produce $J_X$. 
As described in Section \ref{sec:intersection}, $M$ is a complete intersection in $X$ given by the intersection of one hyperplane of class $\PD(p)$ by 5 of $\PD(\xi)$. 
The hypergeometric modification of $J_X=\sum_{a,b}{J_{a,b}(t,z)q_1^aq_2^b}$ is given by
$$I_{X,M}(t,z)=\sum_{a,b}{J_{a,b}(t,z)q_1^aq_2^b\prod_{k=1}^{a}{(p+kz)}\prod_{k=1}^{b}{(\xi+kz)^5}}$$
Restirciting our attention to the identity component $J_X^0=\sum_{i,j}c_{i,j}q_1^i q_2^j$, we have  $I_{X,M}^0=(I_{X,M},\phi^0)=\sum_{i,j}d_{i,j}q_1^i q_2^j$, with $d_{i,j}=c_{i,j}i!(j!)^5$

If we set the degree of $deg(z)=1$ and $deg(t^i)=1-deg(\phi_i)$, $J_X$ is known to be homogeneous of degree 1. 
One can see that in our example the only possible contributions to the mirror map come from the identity component, $J_X^0$, and in particular we need only consider $c_{0,0},c_{1,0},c_{0,1}$. 
We find that $F(t)=1$, $G^0(q,t)=1+q_1$, $G^1(t)=G^2(t)=1$.
\begin{lem}
The mirror map is given by;
$$J_{M,X}(\tau_0,t_1,t_2,z)=I_{M,X}(t_0,t_1,t_2,z)$$
 where $\tau_0=t_0+q_1$. We can more conveniently write this as 
 $$J_{M,X}(t_0,t_1,t_2,z)=e^{-q_1}I_{M,X}(t_0,t_1,t_2,z)$$
 \end{lem}

To produce the quantum period sequence from $J^0_M$ we restrict $t$ to the anti-canonical direction in $H^2(M;\C)$.
As previously stated $-K_M=p+\xi$ and so restricting to the anti-canonical direction has the effect of setting $q_1=q_2=q=e^{t}$.
The effect on $J_M^0$ is to collapse the sum to a power series in one variable with coefficients $d_{i}=\sum_{j+k=i}{c_{j,k}}$.
The first ten terms in the period sequence are: 1, 0, 10, 42, 414, 3300, 29890, 275940, 2608270, 25305000

Since it is known that $J^0_M$ satisfies quantum differential equations, the period sequence also does.
Given sufficient entries in the power series of the period sequence, we may find the differential operator which annihilates the sequence.

The Picard--Fuchs operator for the regularised period sequence of  $M$ is given by
\begin{multline}-17727940t^9D^4 - 47452732t^8D^4 - 177279400t^9D^3 - 51239477t^7D^4\\
 -  400876912t^8D^3 - 620477900t^9D^2 - 28719434t^6D^4 - 363088702t^7D^3\\
 - 1218943172t^8D^2 - 886397000t^9D - 8782543t^5D^4 - 169273876t^6D^3\\
    - 958664473t^7D^2
     - 1562482112t^8D - 425470560t^9 - 1322684t^4D^4\\
      -
    42555106t^5D^3 - 384463114t^6D^2 - 1102964660t^7D - 696963120t^8\\
     -
    37187t^3D^4 - 5281118t^4D^3 - 80112855t^5D^2 - 392394560t^6D -
    456149412t^7\\
     + 13026t^2D^4 - 238966t^3D^3 - 7132816t^4D^2 -
    69331328t^5D - 148485888t^6\\
     + 995tD^4 - 11442t^2D^3 - 3879t^3D^2 -
    4318688t^4D - 22881836t^5 - 24D^4\\ - 2278tD^3
     + 16030t^2D^2 +
    146332t^3D - 928456t^4 + 24D^3 + 35tD^2 + 9600t^2D\\ + 76072t^3 +
    1920t^2\end{multline}
where $D=t\frac{d}{dt} $

This matches the Picard--Fuchs operator predicted by the mirror polytope \cite{Fanosearch}.

\appendix

\section{Quantum multiplication matrices}\label{sec:MMatrix}
\begin{landscape}

$M_p$
\[
\tiny
\arraycolsep=3pt
\left(\begin{array}{*{30}c}
 0&0&0&0&0&0&0&0
&0&0&0&0&0&0&0&0&0&0&0&0&0&0&0&0&q_{{1}}q_{{2}}&3\,q_{{1}}q_{{2}}&5\,q
_{{1}}q_{{2}}&0&0&2\,{q_{{1}}}^{2}q_{{2}}\\ 1&0&0&-q
_{{1}}&0&0&0&0&0&0&0&0&0&0&0&0&0&0&0&0&0&0&0&0&0&0&0&q_{{1}}q_{{2}}&3
\,q_{{1}}q_{{2}}&0\\ 0&0&0&2\,q_{{1}}&0&0&0&0&0&0&0&0
&0&0&0&0&0&0&0&0&0&0&0&0&0&0&0&q_{{1}}q_{{2}}&4\,q_{{1}}q_{{2}}&0
\\ 0&1&0&0&0&0&0&0&0&0&0&0&0&0&0&0&0&0&0&0&0&0&0&0&0
&0&0&0&0&q_{{1}}q_{{2}}\\ 0&0&1&0&0&0&-q_{{1}}&-q_{{
1}}&0&0&0&0&0&0&0&0&0&0&0&0&0&0&0&0&0&0&0&0&0&0\\ 0&0
&0&0&0&0&2\,q_{{1}}&2\,q_{{1}}&0&0&0&0&0&0&0&0&0&0&0&0&0&0&0&0&0&0&0&0
&0&q_{{1}}q_{{2}}\\ 0&0&0&1&0&0&0&0&0&0&-q_{{1}}&0&0
&0&0&0&0&0&0&0&0&0&0&0&0&0&0&0&0&0\\ 0&0&0&0&1&0&0&0
&0&0&q_{{1}}&0&0&0&0&0&0&0&0&0&0&0&0&0&0&0&0&0&0&0
\\ 0&0&0&0&0&1&0&0&0&0&-q_{{1}}&-q_{{1}}&-q_{{1}}&0&0
&0&0&0&0&0&0&0&0&0&0&0&0&0&0&0\\ 0&0&0&0&0&0&0&0&0&0
&q_{{1}}&2\,q_{{1}}&2\,q_{{1}}&0&0&0&0&0&0&0&0&0&0&0&0&0&0&0&0&0
\\ 0&0&0&0&0&0&1&0&0&0&0&0&0&0&0&0&0&0&0&0&0&0&0&0&0
&0&0&0&0&0\\ 0&0&0&0&0&0&0&1&0&0&0&0&0&0&0&-q_{{1}}&0
&0&0&0&0&0&0&0&0&0&0&0&0&0\\ 0&0&0&0&0&0&0&0&1&0&0&0
&0&0&0&q_{{1}}&0&0&0&0&0&0&0&0&0&0&0&0&0&0\\ 0&0&0&0
&0&0&0&0&0&1&0&0&0&0&0&-q_{{1}}&-q_{{1}}&-q_{{1}}&0&0&0&0&0&0&0&0&0&0&0
&0\\ 0&0&0&0&0&0&0&0&0&0&0&0&0&0&0&q_{{1}}&2\,q_{{1}
}&2\,q_{{1}}&0&0&0&0&0&0&0&0&0&0&0&0\\ 0&0&0&0&0&0&0
&0&0&0&0&1&0&0&0&0&0&0&0&0&0&0&0&0&0&0&0&0&0&0\\ 0&0
&0&0&0&0&0&0&0&0&0&0&1&0&0&0&0&0&0&0&-q_{{1}}&0&0&0&0&0&0&0&0&0
\\ 0&0&0&0&0&0&0&0&0&0&0&0&0&1&0&0&0&0&0&0&q_{{1}}&0
&0&0&0&0&0&0&0&0\\ 0&0&0&0&0&0&0&0&0&0&0&0&0&0&1&0&0
&0&0&0&-q_{{1}}&-q_{{1}}&-q_{{1}}&0&0&0&0&0&0&0\\ 0&0
&0&0&0&0&0&0&0&0&0&0&0&0&0&0&0&0&0&0&q_{{1}}&2\,q_{{1}}&2\,q_{{1}}&0&0
&0&0&0&0&0\\ 0&0&0&0&0&0&0&0&0&0&0&0&0&0&0&0&1&0&0&0
&0&0&0&0&0&0&0&0&0&0\\ 0&0&0&0&0&0&0&0&0&0&0&0&0&0&0
&0&0&1&0&0&0&0&0&0&4\,q_{{1}}&10\,q_{{1}}&10\,q_{{1}}&0&0&0
\\ 0&0&0&0&0&0&0&0&0&0&0&0&0&0&0&0&0&0&1&0&0&0&0&0&-
4\,q_{{1}}&-10\,q_{{1}}&-10\,q_{{1}}&0&0&0\\ 0&0&0&0
&0&0&0&0&0&0&0&0&0&0&0&0&0&0&0&1&0&0&0&0&2\,q_{{1}}&5\,q_{{1}}&5\,q_{{
1}}&0&0&0\\ 0&0&0&0&0&0&0&0&0&0&0&0&0&0&0&0&0&0&0&0&0
&1&0&0&0&0&0&10\,q_{{1}}&25\,q_{{1}}&0\\ 0&0&0&0&0&0
&0&0&0&0&0&0&0&0&0&0&0&0&0&0&0&0&1&0&0&0&0&-6\,q_{{1}}&-15\,q_{{1}}&0
\\ 0&0&0&0&0&0&0&0&0&0&0&0&0&0&0&0&0&0&0&0&0&0&0&1&0
&0&0&2\,q_{{1}}&5\,q_{{1}}&0\\ 0&0&0&0&0&0&0&0&0&0&0
&0&0&0&0&0&0&0&0&0&0&0&0&0&0&1&0&0&0&0\\ 0&0&0&0&0&0
&0&0&0&0&0&0&0&0&0&0&0&0&0&0&0&0&0&0&0&0&1&0&0&0\\ 0
&0&0&0&0&0&0&0&0&0&0&0&0&0&0&0&0&0&0&0&0&0&0&0&0&0&0&0&1&0  
\end{array}
\right)
\]

$M_\xi$
\[
\tiny
\begin{pmatrix}
0&0&0&0&0&0&0&0
&0&0&0&0&0&0&0&0&0&0&0&q_{{2}}&0&0&0&0&q_{{1}}q_{{2}}&3\,q_{{1}}q_{{2}
}&5\,q_{{1}}q_{{2}}&0&0&{q_{{1}}}^{2}q_{{2}}\\ 0&0&0
&0&0&0&0&0&0&0&0&0&0&0&0&0&0&0&0&0&0&0&0&q_{{2}}&0&0&0&q_{{1}}q_{{2}}&
3\,q_{{1}}q_{{2}}&0\\ 1&0&0&0&0&0&0&0&0&0&0&0&0&0&0&0
&0&0&0&0&0&0&0&0&0&0&0&q_{{1}}q_{{2}}&4\,q_{{1}}q_{{2}}&0
\\ 0&0&0&0&0&0&0&0&0&0&0&0&0&0&0&0&0&0&0&0&0&0&0&0&0
&0&q_{{2}}&0&0&q_{{1}}q_{{2}}\\ 0&1&0&0&0&0&0&0&0&0&0
&0&0&0&0&0&0&0&0&0&0&0&0&0&0&0&0&0&0&0\\ 0&0&1&0&0&0
&0&0&0&0&0&0&0&0&0&0&0&0&0&0&0&0&0&0&0&0&0&0&0&q_{{1}}q_{{2}}
\\ 0&0&0&0&0&0&0&0&0&0&0&0&0&0&0&0&0&0&0&0&0&0&0&0&0
&0&0&0&q_{{2}}&0\\ 0&0&0&1&0&0&0&0&0&0&0&0&0&0&0&0&0
&0&0&0&0&0&0&0&0&0&0&0&0&0\\ 0&0&0&0&1&0&0&0&0&0&0&0
&0&0&0&0&0&0&0&0&0&0&0&0&0&0&0&0&0&0\\ 0&0&0&0&0&1&0
&0&0&0&0&0&0&0&0&0&0&0&0&0&0&0&0&0&0&0&0&0&0&0\\ 0&0
&0&0&0&0&0&0&0&0&0&0&0&0&0&0&0&0&0&0&0&0&0&0&0&0&0&0&0&q_{{2}}
\\ 0&0&0&0&0&0&1&0&0&0&0&0&0&0&0&0&0&0&0&0&0&0&0&0&0
&0&0&0&0&0\\ 0&0&0&0&0&0&0&1&0&0&0&0&0&0&0&0&0&0&0&0
&0&0&0&0&0&0&0&0&0&0\\ 0&0&0&0&0&0&0&0&1&0&0&0&0&0&0
&0&0&0&0&0&0&0&0&0&0&0&0&0&0&0\\ 0&0&0&0&0&0&0&0&0&1
&0&0&0&0&0&0&0&0&0&0&0&0&0&0&0&0&0&0&0&0\\ 0&0&0&0&0
&0&0&0&0&0&1&0&0&0&0&0&0&0&0&0&0&0&0&0&0&0&0&0&0&0
\\ 0&0&0&0&0&0&0&0&0&0&0&1&0&0&0&0&0&0&0&0&0&0&0&0&0
&0&0&0&0&0\\ 0&0&0&0&0&0&0&0&0&0&0&0&1&0&0&0&0&0&0&0
&0&0&0&0&0&0&0&0&0&0\\ 0&0&0&0&0&0&0&0&0&0&0&0&0&1&0
&0&0&0&0&0&0&0&0&0&0&0&0&0&0&0\\ 0&0&0&0&0&0&0&0&0&0
&0&0&0&0&1&0&0&0&0&0&0&0&0&0&0&0&0&0&0&0\\ 0&0&0&0&0
&0&0&0&0&0&0&0&0&0&0&1&0&0&0&0&0&0&0&0&0&0&0&0&0&0
\\ 0&0&0&0&0&0&0&0&0&0&0&0&0&0&0&0&1&0&0&5&0&0&0&0&0
&0&0&0&0&0\\ 0&0&0&0&0&0&0&0&0&0&0&0&0&0&0&0&0&1&0&-
5&0&0&0&0&0&0&0&0&0&0\\ 0&0&0&0&0&0&0&0&0&0&0&0&0&0&0
&0&0&0&1&3&0&0&0&0&0&0&0&0&0&0\\ 0&0&0&0&0&0&0&0&0&0
&0&0&0&0&0&0&0&0&0&0&1&0&0&5&0&0&0&0&0&0\\ 0&0&0&0&0
&0&0&0&0&0&0&0&0&0&0&0&0&0&0&0&0&1&0&-5&0&0&0&0&0&0
\\ 0&0&0&0&0&0&0&0&0&0&0&0&0&0&0&0&0&0&0&0&0&0&1&3&0
&0&0&0&0&0\\ 0&0&0&0&0&0&0&0&0&0&0&0&0&0&0&0&0&0&0&0
&0&0&0&0&1&0&-5&0&0&0\\ 0&0&0&0&0&0&0&0&0&0&0&0&0&0&0
&0&0&0&0&0&0&0&0&0&0&1&3&0&0&0\\ 0&0&0&0&0&0&0&0&0&0
&0&0&0&0&0&0&0&0&0&0&0&0&0&0&0&0&0&1&3&0
\end{pmatrix}
\]
\end{landscape}

\bibliographystyle{plain}

\end{document}